\documentclass{amsart}

\usepackage{amsmath,amsthm,amssymb}
\usepackage{enumitem}
\usepackage[T1]{fontenc}
\usepackage{graphics}
\usepackage[colorlinks,linkcolor=blue]{hyperref}
\usepackage{mathtools}
\usepackage{MnSymbol}
\usepackage{comment}

\usepackage{tikz}
\usepackage[all]{xy}

\newtheorem{theorem}{Theorem}[section]

\newtheorem{lemma}[theorem]{Lemma}
\newtheorem{proposition}[theorem]{Proposition}

\theoremstyle{definition}

\theoremstyle{remark}

\numberwithin{equation}{section}
\numberwithin{figure}{section}

\newcommand{\XX}{\mathcal{X}}

\newcommand{\NN}{\mathbf{N}}
\newcommand{\ZZ}{\mathbf{Z}}

\newcommand{\ff}{\mathbf{f}}
\newcommand{\kk}{\mathbf{k}}
\newcommand{\nn}{\mathbf{n}}
\newcommand{\mm}{\mathbf{m}}

\begin{document}
\title{Fibonacci numbers and a metric on coprime pairs}
\author{Mitsuaki Kimura}
\address{Department of Mathematics, Osaka Deltal Univercity, Japan}
\email{kimura-m@cc.osaka-dent.ac.jp} 

\begin{abstract}
In this paper, we introduce a metric on the set of pairs of coprime natural numbers. 
We explicitly construct a quasi-isometric embedding from the set of natural numbers into this metric space via Fibonacci numbers.
\end{abstract}

\maketitle



\section{Introduction}
Let $\NN = \{1, 2, 3, \dots\}$ be the set of natural numbers.
We consider the set
\[\XX = \{  \{ n_1,n_2\} \subset \NN \mid \gcd(n_1,n_2)=1 \} \]
 of pairs of coprime natural numbers.
Note that $\{1,1\}=\{1\} \in \XX$.
In this paper, we introduce a natural metric on $\XX$, and observe that consecutive pairs of Fibonacci numbers behave well in this space.

The metric on $\XX$ is defined as follows.
For $\nn = \{n_1, n_2\} \in \XX$, the pair $\nn$ generates the additive group $\ZZ$ of integers.  
Hence every $m \in \NN$ can be expressed as $m = n_1 x + n_2 y$ for some $x,y \in \ZZ$. We define
\[
  q_{\nn}(m)
  = \min \{ |x| + |y| \mid m = n_1 x + n_2 y, \; x, y \in \ZZ \}
\]
and set $q_{\nn}(\mm)   = \max\limits_{m \in \mm} \{ q_{\nn}(m) \}$ for $\mm, \nn \in \XX$.
Then $q_{\nn}$ satisfies a submultiplicative inequality (Lemma~\ref{lem:triangle}).  
For $a>1$, we define
\[
  d_a(\mm, \nn)
  = \log_a \bigl( \max \{ q_{\nn}(\mm), q_{\mm}(\nn) \} \bigr).
\]
Then $d_a$ is a metric on $\XX$ (Proposition~\ref{prop:metric}).

Our definition of the metric $d_a$ is motivated by Tsuboi's construction \cite{Tsuboi}; see also \cite{Kodama, Ishida, Relative}.
Note that $d_a$ can be regarded as the restriction of the metric $\hat\varrho$ defined in \cite[Definition 3.4]{Yagasaki} to the case where $G=\ZZ$.
The metric space $(\XX,d_a)$ will be further studied in the forthcoming work \cite{Filtered}.

We briefly recall the notion of \emph{quasi-isometry}, a fundamental concept in large-scale geometry. For further details, see \cite{GGT} for example.
Let $K \ge 1$ and $L \ge 0$.  
A map $f \colon (X, d_X) \to (Y, d_Y)$ between metric spaces is called a \emph{$(K,L)$-quasi-isometric embedding} if, for all $x_1, x_2 \in X$,
\[
  \frac{1}{K} d_X(x_1, x_2) - L
  \; \le \; 
  d_Y\bigl(f(x_1), f(x_2)\bigr)
  \; \le \;
  K d_X(x_1, x_2) + L.
\]
We say that $X$ and $Y$ are \emph{quasi-isometric} if there exist quasi-isometric embeddings from $X$ to $Y$ and from $Y$ to $X$.
For $a, b > 1$, we have $d_a(\mm, \nn) = (\log_a b) d_b(\mm, \nn)$,
and hence $(\XX, d_a)$ and $(\XX, d_b)$ are quasi-isometric.
Thus, the quasi-isometry class of the metric space $(\XX, d_a)$ does not depend on the choice of $a>1$.  

Let $F_n$ denote the $n$-th Fibonacci number. 
Since consecutive Fibonacci numbers are coprime, we have $\{F_n, F_{n+1}\} \in \XX$  for $n\in \NN$.  
Let $\varphi = \frac{1 + \sqrt{5}}{2}$ denote the golden ratio, and set $d = d_{\varphi}$.  
Our main result is the following:

\begin{theorem}\label{thm:main}
The map 
\[
  (\NN, |\cdot|) \longrightarrow (\XX, d),
  \qquad
  n \longmapsto \{F_{n}, F_{n+1}\},
\]
is a $(1,1)$-quasi-isometric embedding.
\end{theorem}

The paper is organized as follows.
Section~\ref{sec:prelim} includes a proof that $d_a$ defines a metric on $\XX$, and recalls classical properties of Fibonacci numbers.
Section~\ref{sec:main} contains the proof of Theorem~\ref{thm:main}.
Section~\ref{sec:general} discusses further generalizations of Theorem~\ref{thm:main}.

\section{Preliminaries} \label{sec:prelim}

\subsection{A Metric on Coprime Pairs} \label{sec:metric}
We prove that $d_a$ defines a metric on $\XX$. Although this follows from the general framework in \cite{Yagasaki, Filtered}, we include a proof for completeness.

\begin{lemma}\label{lem:triangle}
For all $\kk, \mm, \nn \in \XX$, we have
\[
  q_{\nn}(\kk)
  \le
  q_{\mm}(\kk)\, q_{\nn}(\mm).
\]
\end{lemma}

\begin{proof}
Let $\kk = \{k_1,k_2\}$, $\mm = \{m_1,m_2\}$ and $\nn = \{n_1,n_2\}$, and set
$  M = q_{\mm}(\kk)$, $N = q_{\nn}(\mm)$. 
Since $q_{\mm}(k_1) \le M$ and $q_{\mm}(k_2) \le M$, there exist integers $x_1,y_1,x_2,y_2 \in \ZZ$ such that
\begin{align*}
  k_1 = m_1 x_1 + m_2 y_1, & \qquad |x_1| + |y_1| \le M, \\
  k_2 = m_1 x_2 + m_2 y_2, & \qquad |x_2| + |y_2| \le M.
\end{align*}
Similarly, since $q_{\nn}(m_1) \le N$ and $q_{\nn}(m_2) \le N$, there exist integers
$x'_1,y'_1,x'_2,y'_2 \in \ZZ$ such that
\begin{align*}
  m_1 = n_1 x'_1 + n_2 y'_1, & \qquad |x'_1| + |y'_1| \le N, \\
  m_2 = n_1 x'_2 + n_2 y'_2, & \qquad |x'_2| + |y'_2| \le N.
\end{align*}

Combining these, we have
\[
\begin{bmatrix}
  k_1 \\ k_2
\end{bmatrix}
=
\begin{bmatrix}
  x_1 & y_1 \\
  x_2 & y_2
\end{bmatrix}
\begin{bmatrix}
  x'_1 & y'_1 \\
  x'_2 & y'_2
\end{bmatrix}
\begin{bmatrix}
  n_1 \\ n_2
\end{bmatrix}.
\]
Hence
\begin{align*}
  q_{\nn}(k_1)
  &\le
  |x_1 x'_1 + y_1 x'_2|
  + |x_1 y'_1 + y_1 y'_2| \\
  &\le |x_1|(|x'_1| + |y'_1|)
      + |y_1|(|x'_2| + |y'_2|) \\
  &\le (|x_1| + |y_1|)\, N \\
  &\le MN.
\end{align*}
The same argument applies to $k_2$; hence,
$  q_{\nn}(\kk) \le MN = q_{\mm}(\kk)\, q_{\nn}(\mm). $
\end{proof}

\begin{proposition} \label{prop:metric}
For $a>1$, $d_a$ is a metric on $\XX$.
\end{proposition}

\begin{proof}
The triangle inequality follows from Lemma~\ref{lem:triangle}, and symmetry is immediate from the definition. It remains to check non-degeneracy.  Observe that $q_{\nn}(m)=1$ if and only if $m \in \nn$; hence $q_{\nn}(\mm)=1$ if and only if $\mm \subset \nn$. Thus $d_a(\mm,\nn)=0$ holds precisely when $\mm=\nn$. Therefore, $d_a$ is a metric.
\end{proof}

\subsection{Fibonacci Numbers} \label{sec:fibonacci}
We recall classical properties of Fibonacci numbers for later use.
The Fibonacci sequence $(F_n)_{n \ge 0}$ is defined by the recurrence relation
\[
  F_0 = 0,\quad F_1 = 1,\qquad F_{n+2} = F_{n+1} + F_n.
\]
This definition naturally extends to negative indices, yielding a bi-infinite sequence. In particular,  for all $n \in \NN$,
\[
  F_{-n} = (-1)^{n+1} F_n.
\]

\begin{lemma}[Honsberger's identity] \label{lem:honsberger}
For all $m, n \in \ZZ$, the following holds:
\[
  F_{m+n} = F_m F_{n+1} + F_{m-1} F_n.
\]
\end{lemma}

\begin{proof}
Set $A =
\begin{bmatrix}
  1 & 1 \\ 1 & 0
\end{bmatrix}.$
For all $n \in \NN$, we can verifies by induction that
\[
A^n =
\begin{bmatrix}
  F_{n+1} & F_n \\ F_n & F_{n-1}
\end{bmatrix}.
\]
Using the identity $F_{-n} = (-1)^{n+1} F_n$, the same formula holds for all $n \in \ZZ$. Comparing the entries of $A^{m+n} = A^m A^n$ yields the desired identity.
\end{proof}

\begin{lemma}\label{lem:golden_ratio}
Let $\varphi = \frac{1+\sqrt{5}}{2}$ denote the golden ratio. For all $n \in \NN$, we have
\[
  \varphi^{n-2} \le F_n \le \varphi^{n-1}.
\]
\end{lemma}

\begin{proof}
The proof proceeds by induction, using the identity $\varphi^2 = \varphi + 1$.
\end{proof}
\section{Main Result} \label{sec:main}

In this section, we prove Theorem~\ref{thm:main}. For $n \in \NN$, set $\ff_n = \{F_n,F_{n+1}\}$. 
The statement of Theorem~\ref{thm:main} concerns the estimation of
\[
d(\ff_m,\ff_n) = \log_\varphi \bigl( \max\{ q_{\ff_n}(\ff_m), q_{\ff_m}(\ff_n) \} \bigr),
\]
and thus the proof reduces to bounding 
\[
q_{\ff_n}(F_m) = \min\{ |x|+|y| \mid F_m = xF_n + yF_{n+1},\; x,y \in \ZZ \}.
\]
 from above and below. 
 Accordingly, we study the Diophantine equation
\begin{equation}\label{eq:diophantus}
F_m = xF_n + yF_{n+1}.
\end{equation}


\begin{lemma}[Lower bound]\label{lem:lower}
Let $m,n \in \NN$ with $m>n$. Then
\[
q_{\ff_n}(F_m) \ge \frac{F_m}{F_{n+1}}.
\]
\end{lemma}

\begin{proof}
If $(x,y)$ is a solution of \eqref{eq:diophantus}, then
\[
F_m = |xF_n + yF_{n+1}| \le |x|F_n + |y|F_{n+1} \le (|x|+|y|)F_{n+1},
\]
which yields the desired inequality.
\end{proof}

\begin{lemma}[Upper bound]\label{lem:upper}
Let $m,n \in \NN$. Then
\[
q_{\ff_n}(F_m) \le |F_{m-n-1}| + |F_{m-n}| =
\begin{cases}
F_{m-n+1}, & \text{if $m \ge n$},\\
F_{n-m+2}, & \text{if $m \le n$}.
\end{cases}
\]
\end{lemma}

\begin{proof}
By Lemma~\ref{lem:honsberger}, the pair $(x,y) = (F_{m-n-1}, F_{m-n})$ provides a solution to \eqref{eq:diophantus}. Combining this with the identity $F_{-n} = (-1)^{n+1} F_n$ yields the desired result.
\end{proof}

\begin{proof}[Proof of Theorem~{\ref{thm:main}}]
For $m,n \in \NN$, we prove that
\begin{equation} \label{eq:QI}
|m-n| -1 \leq d( \ff_m, \ff_n ) \leq |m-n| +1.
\end{equation}

The case $m=n$ is trivial. 
We may assume without loss of generality that $m>n$.  
By Lemmas~\ref{lem:golden_ratio} and~\ref{lem:lower}, we obtain
\[
\log_\varphi q_{\ff_n}(F_m)  \geq \log_\varphi  F_m - \log_\varphi F_{n+1} \geq (m-2)-n.
\]
Together with $\log_\varphi q_{\ff_n}(F_{m+1}) \geq m-n-1$, this implies
\begin{equation} \label{eq:lower}
\log_\varphi q_{\ff_n}(\ff_m) \geq m-n-1.
\end{equation}

On the other hand, by Lemmas~\ref{lem:golden_ratio} and~\ref{lem:upper},
\[
\log_\varphi q_{\ff_n}(F_m) \leq \log_\varphi F_{m-n+1} \leq m-n
\]
holds. Hence, together with $\log_\varphi q_{\ff_n}(F_{m+1}) \leq m-n+1$, we obtain
\begin{equation} \label{eq:m>n}
\log_\varphi q_{\ff_n}(\ff_m) \leq m-n+1.
\end{equation}

Similarly, from
\[
\log_\varphi q_{\ff_m}(F_n) \leq \log_\varphi F_{m-n+2} \leq m-n+1
\]
and $\log_\varphi q_{\ff_m}(F_{n+1}) \leq m-n$, we obtain
\begin{equation} \label{eq:m<n}
\log_\varphi q_{\ff_m}(\ff_n) \leq m-n+1.
\end{equation}

Combining \eqref{eq:lower}, \eqref{eq:m>n}, and \eqref{eq:m<n}, we conclude that
\[
 m-n-1
 \leq \log_\varphi q_{\ff_m}(\ff_n)
 \leq \log_\varphi \bigl( \max\{  q_{\ff_n}(\ff_m),  q_{\ff_m}(\ff_n) \}\bigr)
 \leq m-n+1,
\]
which proves~\eqref{eq:QI}.
\end{proof}

\section{Generalizations} \label{sec:general}

\subsection{$k$-Fibonacci Numbers}
For $k \in \NN$,  the \emph{$k$-Fibonacci sequence} $(F_{k,n})_{n \in \ZZ}$ is defined as in \cite{k-Fibonacci} by the recurrence
\[
  F_{k,0} = 0, \quad F_{k,1} = 1, \qquad
  F_{k,n+2} = k F_{k,n+1} + F_{k,n}.  
\]
It satisfies $F_{k,-n}=(-1)^{n+1}F_{k,n}$ for all $n \in \NN$.
Note that $F_{1,n}$ denotes the standard Fibonacci numbers, and $F_{2,n}$ denotes the Pell numbers.

The positive root of the quadratic equation $x^2 = kx + 1$,
\[
  \varphi_k := \frac{k + \sqrt{k^2 + 4}}{2},
\]
is referred to as the \emph{$k$-th metallic ratio}. For example, $\varphi_1=\frac{1+\sqrt{5}}{2}$ is the golden ratio, and $\varphi_2=1+\sqrt{2}$ is the silver ratio.

We now generalize Theorem~\ref{thm:main}; setting $k=1$ yields the original statement.
\begin{theorem}\label{thm:k-main}
Let $k \in \NN$. The map
\[
  (\NN, |\cdot|) \longrightarrow (\XX, d_{\varphi_k}),
  \qquad
  n \longmapsto \{F_{k,n}, F_{k,n+1}\},
\]
is a $(1,1)$-quasi-isometric embedding.
\end{theorem}

To prove Theorem~\ref{thm:k-main}, we generalize the lemmas from Subsection~\ref{sec:fibonacci}.

\begin{lemma}[{\cite[Proposition 14]{k-Fibonacci}}] \label{lem:k-honsberger}
Let $k \in \NN$.  
For all $m,n \in \ZZ$, we have
\[
  F_{k,m+n}
  =
  F_{k,m} F_{k,n+1}
  + F_{k,m-1} F_{k,n}.
\]
\end{lemma}

\begin{proof}
For $n \in \ZZ$, we have
\[
  \begin{bmatrix}
    k & 1 \\
    1 & 0
  \end{bmatrix}^n
  =
  \begin{bmatrix}
    F_{k,n+1} & F_{k,n} \\
    F_{k,n} & F_{k,n-1}
  \end{bmatrix}.
\]
The identity follows by comparing the entries of $A^{m+n}=A^mA^n$.
\end{proof}


\begin{lemma}\label{lem:metalic_ratio}
Let $k \in \NN$.  For all $n \in \NN$, the following inequalities hold:

\begin{enumerate}[label=$(\arabic*)$]
\item  $k \varphi_k^{n-2}
  \le
  F_{k,n}
  \le
  k \varphi_k^{n-1}$
\item
  $\varphi_k^{n-2}
  \le
  F_{k,n-2} + F_{k,n-1}
  \le
  \varphi_k^{n-1}$
\end{enumerate}
\end{lemma}
\begin{proof}
Both statements follow by induction, using the identity $\varphi_k^2=k\varphi_k+1$.
\end{proof}

As in the proof of Theorem \ref{thm:main}, we consider the Diophantine equation
\[
  F_{k,m}
  =
  x F_{k,n} + y F_{k,n+1}.
\]
By Lemma~\ref{lem:k-honsberger}, the pair
\(
  (x,y) = (F_{k,m-n-1},\, F_{k,m-n})
\)
satisfies this equation.

Let $\ff_{k,n} = \{F_{k,n}, F_{k,n+1}\}$. 
The following generalizations of Lemmas~\ref{lem:lower} and~\ref{lem:upper} hold; the proofs are analogous and omitted.
\begin{lemma} \label{lem:lower_general}
 Let $k \in \NN$. For all $m,n \in \NN$ satisfy $m > n$, we have
\[
 q_{\ff_{k,n} }(F_{k,m}) \geq \frac{F_{k,m}}{F_{k, n+1}}.
\]
\end{lemma}

\begin{lemma} \label{lem:upper_general}
 Let $k \in \NN$. For all $m,n \in \NN$, we have
\[
  q_{\ff_{k,n}}(F_{k,m})
  \le
  |F_{k,m-n-1}| + |F_{k,m-n}|
  =
 \begin{cases}
 F_{k,m-n-1} + F_{k,m-n} & \text{if $m \geq n$}, \\
 F_{k,n-m} + F_{k,n-m+1} & \text{if $m \leq n$}.
 \end{cases}
\]
\end{lemma}

\begin{proof}[Proof of Theorem~{\ref{thm:k-main}}]
The proof parallels that of Theorem~\ref{thm:main}.
We prove that
\[
|m-n| -1 \leq d( \ff_{k,m}, \ff_{k,n} ) \le |m-n| +1.
\]
for all $m,n \in \NN$. 
We may assume $m>n$.  
By Lemmas~\ref{lem:metalic_ratio} (1) and~\ref{lem:lower_general}, we have
\[
\log_\varphi q_{\ff_{k,n}}(\ff_{k,m}) \ge m-n-1.
\]
By Lemmas~\ref{lem:metalic_ratio} (2) and~\ref{lem:upper_general}, we have
\[
\log_\varphi q_{\ff_{k,n}}(\ff_{k,m}) \le m-n+1 
\quad \text{and} \quad 
\log_\varphi q_{\ff_{k,m}}(\ff_{k,n}) \le m-n+1.
\]
Combining these yields the desired inequality.
\end{proof}

\subsection{$\ell$-Coprime Numbers}
For $\ell \geq 2$, we consider the set of $\ell$-coprime numbers
\[\XX^\ell = \{  \{ n_1,\dots, n_\ell\} \subset \NN \mid \gcd(n_1,\dots, n_\ell)=1 \}. \]
Note that $\XX^2=\XX$. We extend the metric $d$ on $\XX$ to $\XX^\ell$ as follows. 

For $\nn=\{n_1,\dots,n_\ell\}\in\XX^\ell$ and  $m\in \NN$, define
\[
  q_{\nn}(m)
  = \min \{ |x_1| + \dots + |x_\ell| \mid m = n_1 x_1 + \dots + n_\ell x_\ell, \; x_1, \dots, x_\ell \in \mathbb{Z} \}.
\]
For $\mm, \nn \in \XX^\ell$, set $q_{\nn}(\mm)   = \max\limits_{m \in \mm} \{ q_{\nn}(m) \}$, and define
\[
  d^\ell(\mm, \nn)
  = \log_{\varphi} \bigl( \max \{ q_{\nn}(\mm),\, q_{\mm}(\nn) \} \bigr).
\]
A similar argument as in Subsection \ref{sec:metric} shows that $d^\ell$ is a metric on $X^\ell$.

The following theorem provides a generalization of Theorem~\ref{thm:main}, which is recovered as the special case $\ell=2$.

\begin{theorem}\label{thm:l-main}
Let $\ell \geq 2$. The map
\[
  (\NN, |\cdot|) \longrightarrow (\XX^\ell, d^\ell),
  \qquad
  n \longmapsto \{F_{n}, F_{n+1}, \dots, F_{n+\ell-1}\},
\]
is a $(1,1)$-quasi-isometric embedding.
\end{theorem}

\begin{proof}The argument parallels that of Theorem~\ref{thm:main}. 
Let $\ff_{n}^\ell = \{F_{n}, F_{n+1}, \dots, F_{n+\ell-1}\}$. We prove that
\begin{equation} \label{eq:QI_general}
|m-n| -1 \leq d^\ell ( \ff_{m}^\ell , \ff_{n}^\ell ) \le |m-n| +1.
\end{equation}
for all $m,n \in \NN$. 
We may assume that $m>n$.  
Let $i \in \{0,1,\dots,\ell-1\}$.

First we estimate $q_{\ff_{n}^\ell } (F_{m+i})$. 
Assume that $m+i \ge n + \ell$; otherwise $q_{\ff_n^\ell}(F_{m+i})=1$. Then we have
\[
q_{\ff_n^\ell}(F_{m+i})\ge\frac{F_{m+i}}{F_{n+\ell-1}}.
\]
Thus, by Lemma~\ref{lem:golden_ratio}, we obtain
\begin{equation} \label{eq:l-lower}
\log_\varphi q_{\ff_n^\ell}(\ff_m^\ell) =\log_\varphi \left( \max\limits_{0\leq i \leq \ell-1}\{ q_{\ff_n^\ell}(F_{m+i}) \} \right) \ge \log_{\varphi} \frac{F_{m+\ell-1}}{F_{n+\ell-1}} \ge m-n-1.
\end{equation}

Since the Diophantine equation
\[  F_{m+i}  =  x_1 F_{n} + x_2 F_{n+1}+ \dots + x_\ell F_{n+\ell-1} \]
has a solution $ (x_1, \dots, x_\ell) = (0, \dots, 0, F_{m+i-n-\ell+1}, F_{m+i-n-\ell+2})$, we have
\[
 q_{\ff_n^\ell} (F_{m+i}) \leq F_{m+i-n-\ell+1} + F_{m+i-n-\ell+2} = 
 F_{m+i-n-\ell+3} \le F_{m-n+2}.
\]
Hence we have 
\begin{equation} \label{eq:l_m>n}
\log_\varphi q_{\ff_n}(\ff_m) \le \log_\varphi F_{m-n+2} \le m-n+1.
\end{equation}

Next we estimate $q_{\ff_{m}^\ell } (F_{n+i})$. 
Assume that $n+i < m$; otherwise $q_{\ff_{m}^\ell } (F_{n+i})=1$. 
Since the Diophantine equation
\[  F_{n+i}  =  x_1 F_{m} + x_2 F_{m+1}+ \dots + x_\ell F_{m+\ell-1} \]
has a solution
$ (x_1, \dots, x_\ell) = (F_{n+i-m-1}, F_{n+i-m}, 0, \dots, 0)$, we have
\[
 q_{\ff_m}(F_{n+i}) \leq |F_{n+i-m-1}|+|F_{n+i-m}|= F_{m-n-i+1}+F_{m-n-i} = F_{m-n-i+2} \leq F_{m-n+2}
\]
Hence we have 
\begin{equation} \label{eq:l_m<n}
\log_\varphi q_{\ff_m}(\ff_n)  \le \log_\varphi F_{m-n+2} \le m-n+1.
\end{equation}

Combining \eqref{eq:l-lower}, \eqref{eq:l_m>n} and \eqref{eq:l_m<n}, we obtain \eqref{eq:QI_general}.
\end{proof}

\section*{Acknowlegment}
The author would like to thank Morimichi Kawasaki for proposing the problem that led to the main result. 
He also thanks Hiroki Kodama for helpful suggestions regarding Theorem~\ref{thm:l-main}, which enabled the removal of unnecessary assumptions.

The author is partially supported by JSPS KAKENHI Grant Number JP24K16921.

\bibliographystyle{abbrv}
\bibliography{reference}


\end{document}